\documentclass[12pt]{article}

\usepackage[margin=1in]{geometry}
\usepackage{amsthm}
\usepackage{amssymb}
\usepackage{amsmath}
\usepackage{graphicx}
\usepackage{mathdots}
\usepackage{mathtools}
\usepackage{url}
\usepackage{color}
\usepackage{framed}
\usepackage{tikz}
\usepackage{stmaryrd}
\usepackage[ruled,vlined]{algorithm2e}
\usepackage{array}

\newcommand{\llangle}{\langle\hspace{-2.5pt}\langle}
\newcommand{\rrangle}{\rangle\hspace{-2.5pt}\rangle}

\newtheorem{theorem}{Theorem}

\newtheorem{lemma}[theorem]{Lemma}
\newtheorem{proposition}[theorem]{Proposition}

\theoremstyle{definition}

\newtheorem{definition}[theorem]{Definition}

\title{Max filtering with reflection groups}

\author{Dustin~G.~Mixon\footnote{Department of Mathematics, The Ohio State University, Columbus, OH} \footnote{Translational Data Analytics Institute, The Ohio State University, Columbus, OH}
\quad
Daniel~Packer\footnotemark[1]
}

\date{}

\begin{document}
\maketitle

\begin{abstract}
Given a finite-dimensional real inner product space $V$ and a finite subgroup $G$ of linear isometries, max filtering affords a bilipschitz Euclidean embedding of the orbit space $V/G$.
We identify the max filtering maps of minimum distortion in the setting where $G$ is a reflection group.
Our analysis involves an interplay between Coxeter's classification and semidefinite programming.
\end{abstract}

\section{Introduction}

Machine learning often requires an object to be represented as a point $x$ in a vector space $V$.
In many settings, there is a linear group $G$ of ambiguities for which the same object can be represented as $gx$ for any $g\in G$.
For example, when the object is a point cloud or a graph, the vector representation depends on how the points or vertices are labeled.
One is inclined to apply a $G$-invariant feature map $\Phi\colon V\to F$ to factor out these ambiguities before training a machine learning model.
Recently, \cite{CahillIMP:22}~introduced such a feature map that is well suited for groups of linear isometries:

\begin{definition}
Fix a real inner product space $V$ and a group $G\leq\operatorname{O}(V)$.
\begin{itemize}
\item[(a)]
The \textbf{max filtering map} $\llangle\cdot,\cdot\rrangle\colon V/G\times V/G\to\mathbb{R}$ is defined by
\[
\llangle G\cdot x,G\cdot y\rrangle
:=\sup_{\substack{p\in G\cdot x\\q\in G\cdot y}}\langle p,q\rangle.
\]
\item[(b)]
Given $\{t_i\}_{i=1}^n$ in $V$, the corresponding \textbf{max filter bank} $\Phi\colon V/G\to\mathbb{R}^n$ is defined by
\[
\Phi(G\cdot x):=\{\llangle G\cdot t_i,G\cdot x\rrangle\}_{i=1}^n.
\]
\end{itemize}
\end{definition}

The vectors $\{t_i\}_{i=1}^n$ that determine a max filter bank are known as \textbf{templates}.
Max filter banks offer a noteworthy contrast to the polynomial invariants that are typically studied in the literature~\cite{BandeiraBKPWW:17,PerryWBRS:19,CahillCC:20,BendoryELS:22,BalanHS:22}.
In particular, every $G$-invariant polynomial map $V\to\mathbb{R}^n$ is either affine linear or not Lipschitz.
As a result, polynomial invariants are doomed to either fail to separate orbits or fail to exhibit numerical stability.
Meanwhile, \cite{CahillIMP:22}~establishes that for every finite $G\leq\operatorname{O}(d)$, there exists a \textit{bilipschitz} max filter bank $\mathbb{R}^d/G\to\mathbb{R}^n$.
We enunciate the quantitative details of this result in terms of the following definition:

\begin{definition}
Given a group $G\leq\operatorname{O}(d)$ whose action on $\mathbb{R}^d$ has closed orbits, the \textbf{optimal max filtering condition number} $\kappa(G)$ is the infimum of quotients $C/c$ for which there exist $n\in\mathbb{N}$ and templates $\{t_i\}_{i=1}^n$ in $\mathbb{R}^d$ such that the corresponding max filter bank (as a map from $\mathbb{R}^d/G$ with the quotient Euclidean metric to $\mathbb{R}^n$ with the Euclidean metric) has upper and lower Lipschitz bounds $C$ and $c$, respectively.
\end{definition}

\begin{proposition}[see Theorem~18 in~\cite{CahillIMP:22}]
\label{prop.general bound on kappa}
There exists a universal constant $C>0$ such that for every $d,m\in\mathbb{N}$ and every finite group $G\leq\operatorname{O}(d)$ of order $m$, it holds that 
\[
\kappa(G)
\leq Cm^3d^{1/2}(d\log d+d\log m+\log^2m)^{1/2}.
\]
\end{proposition}

Proposition~\ref{prop.general bound on kappa} is the first of its kind, but it is not sharp and thus warrants further investigation.
In this paper, we derive the exact value of $\kappa(G)$ for every finite reflection group $G$.
To compare with Proposition~\ref{prop.general bound on kappa}, an easy-to-state consequence of our main results is that
\[
\kappa(G)\leq Cm
\]
whenever $G$ is a finite reflection group.
While this upper bound is saturated by the dihedral groups, $\kappa(G)$ is typically much smaller, frequently being on the order of $d$.
In addition to determining $\kappa(G)$ for each finite reflection group $G$, we also identify $n=d$ max filtering templates that attain this optimal condition number.

In the next section, we review some preliminaries on reflection groups.
In Section~\ref{sec.main results}, we formulate our main results and describe how our problem reduces to the study of Weyl chambers of essential and irreducible finite reflection groups.
Sections~\ref{sec.proof of thm.main} and~\ref{sec.proof of thm.asymptotic} contain the proofs of our main results.

\section{Background on reflection groups}

This section reviews basic information about reflection groups.
We encourage the interested reader to consult~\cite{Kane:01} for more information.

A \textbf{reflection} is an orthogonal transformation with exactly one negative eigenvalue.
Every nonzero vector $u\in\mathbb{R}^d$ determines a reflection with matrix representation
\[
R(u):=I-\frac{2}{\|u\|^2}uu^\top,
\]
where $I$ denotes the identity matrix.
A \textbf{reflection group} is any group that is generated by a set of reflections, implying that it is a subgroup of the orthogonal group $\operatorname{O}(d)$.
For example, the group of permutation matrices is generated by the transposition matrices $R(e_i-e_j)$, where $e_i$ and $e_j$ denote standard basis elements of $\mathbb{R}^d$.
Given nonzero vectors $\{u_i\}_{i=1}^k$, the fixed points of the reflection group generated by $\{R(u_i)\}_{i=1}^k$ form the orthogonal complement of $\operatorname{span}\{u_i\}_{i=1}^k$.
We say a reflection group is \textbf{essential} if the origin is its only fixed point.
The group of permutation matrices fixes the all-ones vector, and so this reflection group is not essential in $\mathbb{R}^d$ (however, it is essential when interpreted as a reflection group on the orthogonal complement of the all-ones vector).

Given a finite reflection group $G\leq\operatorname{O}(d)$, let $U$ denote the union of the hyperplanes fixed by each individual reflection in $G$.
The connected components of $\mathbb{R}^d\setminus U$ are known as the \textbf{Weyl chambers} of $G$.
Fix a Weyl chamber $C\subseteq\mathbb{R}^d$.
It turns out that $C$ is an open set whose boundary is contained in a union of $\ell\in\mathbb{N}$ hyperplanes corresponding to reflections in $G$ (and no fewer).
We denote these reflections by $\{R(\alpha_i)\}_{i=1}^\ell$ for unit vectors $\{\alpha_i\}_{i=1}^\ell$, which are signed so that they reside in the dual cone of $C$.
Then $\{\alpha_i\}_{i=1}^\ell$ is known as the \textbf{fundamental system} of $G$ corresponding to $C$.
Interestingly, the reflections $\{R(\alpha_i)\}_{i=1}^\ell$ generate the original reflection group $G$.
In the case where $G$ is essential, we have $\ell=d$, and furthermore, the dual basis $\{\beta_i\}_{i=1}^d$ of the fundamental system $\{\alpha_i\}_{i=1}^d$ generates the closed convex cone $\overline{C}$, i.e., $\overline{C}$ is an example of a simplicial cone.

As an example, consider the symmetries of the square, i.e., the dihedral group of order~$8$.
Figure~\ref{fig.weyl_chamber} illustrates the $1$-eigenspaces (in this case, lines) of the reflections in this group.
These lines carve the plane into eight Weyl chambers, one of which we call $C$.
The corresponding fundamental system consists of the unit vectors $\alpha_1:=(\frac{1}{\sqrt{2}},-\frac{1}{\sqrt{2}})$ and $\alpha_2:=(0,1)$, and the resulting dual basis, given by $\beta_1:=(\sqrt{2},0)$ and $\beta_2:=(1,1)$, generates $\overline{C}$. 

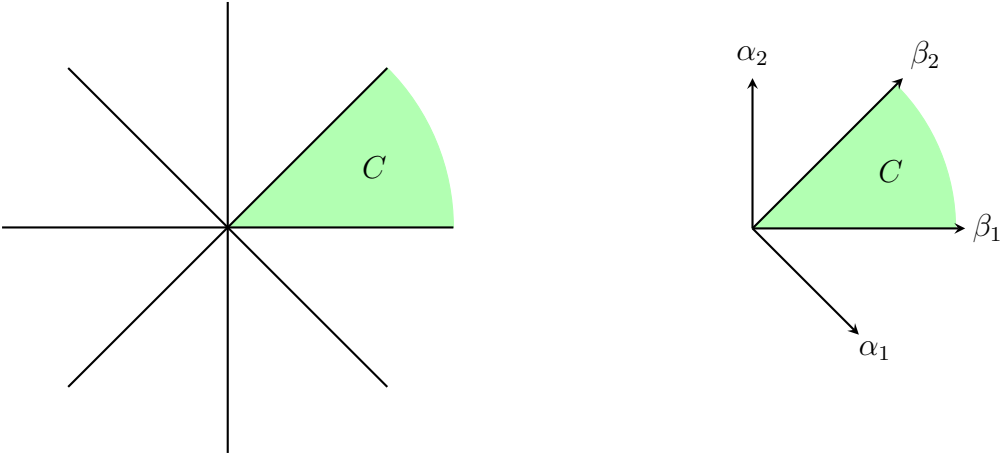
\begin{figure}
\begin{center}
\begin{tikzpicture}[scale=3]
\draw[green!30,fill=green!30] (0,0) --  (45:1) arc(45:0:1) -- cycle;
\draw[thick] (-1,0) -- (1,0);
\draw[thick] (-1/1.414,-1/1.414) -- (1/1.414,1/1.414);
\draw[thick] (0,-1) -- (0,1);
\draw[thick] (1/1.414,-1/1.414) -- (-1/1.414,1/1.414);
\draw (0.65,0.65*0.414) node {$C$};
\end{tikzpicture}
\qquad
\begin{tikzpicture}[scale=2]
\draw[draw=none] (-1.5,0) -- (1.5,0);
\draw[draw=none] (0,-1.5) -- (0,1.5);
\draw[green!30,fill=green!30] (0,0) --  (45:1.35) arc(45:0:1.35) -- cycle;
\draw (0.65*1.414,0.65*0.414*1.414) node {$C$};
\draw [thick, -stealth](0,0) -- (0,1);
\draw (0,1.15) node {$\alpha_2$};
\draw [thick, -stealth](0,0) -- (1/1.414,-1/1.414);
\draw (1.15/1.414,-1.15/1.414) node {$\alpha_1$};
\draw [thick, -stealth](0,0) -- (1,1);
\draw (1.15,1.15) node {$\beta_2$};
\draw [thick, -stealth](0,0) -- (1.414,0);
\draw (1.414+0.15,0) node {$\beta_1$};
\end{tikzpicture}
\end{center}
\caption{\label{fig.weyl_chamber}
\textbf{(left)}
Weyl chamber arising from dihedral group.
\textbf{(right)}
Fundamental system and its dual basis.
}
\end{figure}

A reflection group in $\operatorname{O}(d)$ is said to be \textbf{reducible} if it can be decomposed as a direct product of nontrivial reflection groups in $\operatorname{O}(d)$, in which case $\mathbb{R}^d$ can be correspondingly decomposed into orthogonal invariant subspaces.
For example, the group of orthogonal diagonal matrices is reducible and can be decomposed along the lines spanned by the standard basis.
We note that this use of ``reducible'' differs slightly from its usage in representation theory.
While a reflection group may be viewed as the image of a representation, the representation is irreducible precisely when its image is both irreducible \textit{and} essential as a reflection group.
For example, the group of permutation matrices is reducible as a representation but irreducible as a reflection group  (since it is not essential).

For each $d\in\mathbb{N}$, there are finitely many finite essential irreducible reflection groups in $\operatorname{O}(d)$ up to conjugation by an orthogonal matrix.
These were famously classified by Coxeter~\cite{Coxeter:35} in terms of abstract group presentations.
To motivate these presentations, note that for linearly independent unit vectors $u,v\in\mathbb{R}^d$, the composition $R(u)R(v)$ has a $1$-eigenspace of dimension $d-2$, and it acts as a rotation by $2\arccos\langle u,v\rangle$ radians in the $2$-dimensional orthogonal complement.
If $R(u)$ and $R(v)$ belong to a common finite reflection group, it must hold that $R(u)R(v)$ has finite order, and so $\arccos\langle u,v\rangle$ is a rational multiple of $\pi$.
If $u$ and $v$ belong to a fundamental system of the reflection group, it turns out that $\langle u,v\rangle=-\cos\frac{\pi}{m}$, where $m$ is the order of $R(u)R(v)$.
As an abstraction of this, a \textbf{Coxeter system} consists of generators $\{r_i\}_{i=1}^d$ and orders $\{m_{ij}\}_{i,j=1}^d$ such that $m_{ii}=1$ and $m_{ij}\geq 2$ for $i\neq j$.
The resulting \textbf{Coxeter group} is then given by
\[
\langle r_1,\ldots,r_d:(r_ir_j)^{m_{ij}}=1\rangle.
\]
Every Coxeter system has a \textbf{Coxeter diagram}, namely, a graph with vertex set $\{1,\ldots,d\}$ and adjacency $i\leftrightarrow j$ precisely when $m_{ij}\geq3$, and this edge is labeled by $m_{ij}$ whenever $m_{ij}\geq4$.
The Coxeter system is \textbf{irreducible} when the Coxeter diagram is connected.
For example, the connected Coxeter diagram on two vertices whose edge is labeled by $4$ corresponds to the Coxeter group
\[
\langle r_1,r_2:r_1^2=r_2^2=(r_1r_2)^{4}=1\rangle,
\]
which is isomorphic to the dihedral group of order $8$; see Figure~\ref{fig.weyl_chamber} for a relevant illustration.

Given a finite irreducible Coxeter system, one may use the geometry of fundamental systems to reconstruct the corresponding finite essential irreducible reflection group (which again is unique up to conjugation).
Indeed, the Gram matrix of the fundamental system $\{\alpha_i\}_{i=1}^d$ is given by
\[
\langle \alpha_i,\alpha_j\rangle
=\left\{\begin{array}{cl}1&\text{if }i=j\\
-\cos\frac{\pi}{m_{ij}}&\text{if }i\neq j,
\end{array}\right.
\]
which can then be Cholesky factored to determine $\{\alpha_i\}_{i=1}^d$ up to orthogonal transformation.
Then $\{R(\alpha_i)\}_{i=1}^d$ generates the desired reflection group.
In the above case where $d=2$ and $m_{12}=4$, we obtain the Gram matrix
\[
\left[\begin{array}{cc}
\phantom{-}1&-\frac{1}{\sqrt{2}}\\
-\frac{1}{\sqrt{2}}&\phantom{-}1
\end{array}\right],
\]
which can be factored to recover the fundamental system illustrated in Figure~\ref{fig.weyl_chamber}.

In the case where $G$ is essential and irreducible, the closed fundamental chamber is an \textbf{acute cone}, i.e., a convex cone $K$ such that every nonzero $x,y\in K$ satisfies $\langle x,y\rangle>0$.
This geometric fact will prove useful in understanding the spectral properties associated with fundamental systems of $G$.

\begin{lemma}
\label{lem.acute cone}
For every essential and irreducible finite reflection group $G$, each closed chamber of $G$ is an acute cone.
\end{lemma}

\begin{proof}
Since the rows of $A$ form a fundamental system, the off-diagonal entries of the Gram matrix $AA^\top$ are nonpositive; see Lemma~B on page~39 in~\cite{Kane:01}.
Furthermore, since the reflection group is irreducible, the Gram matrix $AA^\top$ is irreducible as a matrix (i.e., it is not permutation equivalent to a block-diagonal matrix).
It follows from Corollary~3.24 in~\cite{Varga:62} that $B^\top B=(AA^\top)^{-1}$ is entrywise positive, which implies the result.
\end{proof}

All of our main results rely on the following spectral phenomena:

\begin{lemma}
\label{lem.spectral phenomenon}
Suppose $G\leq\operatorname{O}(d)$ is an essential and irreducible finite reflection group, and select any matrix $A\in\mathbb{R}^{d\times d}$ with unit-norm rows that form a fundamental system for $G$.
\begin{itemize}
\item[(a)]
The top and bottom eigenvectors of $AA^\top$ satisfy $(v_1)_i^2=(v_d)_i^2$ for all $i\in[d]$.
\end{itemize}
Next, consider the matrix $M:=AA^\top-I_d$.
For each $j\in[d]$, let $M^{(j)}$ denote the principal submatrix of $M$ obtained by discarding the $j$th row and column, and let $\lambda_k(\cdot)$ denote the $k$th largest eigenvalue of the argument, counted with multiplicity.
\begin{itemize}
\item[(b)]
For each $k\in[d]$, it holds that $\lambda_{d-k+1}(M)=-\lambda_k(M)$.
\item[(c)]
For each $j\in[d]$ and $k\in[d-1]$, it holds that $\lambda_{d-k}(M^{(j)})=-\lambda_k(M^{(j)})$.
\end{itemize}
\end{lemma}

We could not locate a proof of the above facts, and so we provide our own in Appendix~\ref{appendix.spectral phenomenon}.
Our proof argues by cases over Coxeter's classification.

\section{Main results and key reductions}
\label{sec.main results}

Our first main result describes minimal optimal max filter banks for every reflection group in terms of its decomposition into irreducible constituent representations:

\begin{theorem}
\label{thm.main}
Suppose $G\leq\operatorname{O}(d)$ is a finite reflection group.
\begin{itemize}
\item[(a)]
If $G$ is the trivial group, then $\kappa(G)=1$, and minimal  optimal templates are given by any orthonormal basis.
\item[(b)]
If $G$ is essential and irreducible, then for any matrix $A\in\mathbb{R}^{d\times d}$ with unit-norm rows that form a fundamental system for $G$, it holds that $\kappa(G)=\frac{\sigma_{\max}(A)}{\sigma_{\min}(A)}$, and minimal optimal templates are given by the columns of $A^{-1}$.
\item[(c)]
Otherwise, $G$ can be decomposed into irreducible constituents $G_i\leq\operatorname{O}(d_i)$, one of which is trivial precisely when $G$ is not essential, and the others are essential and irreducible as reflection groups.
Then $\kappa(G)=\max_i\kappa(G_i)$, and minimal optimal templates are given by an orthogonal direct sum of the minimal optimal templates for each irreducible constituent described in (a) and (b).
\end{itemize}
\end{theorem}

Theorem~\ref{thm.main}(a) is immediate.
The proof of part~(c) is straightforward after some preliminary analysis of Weyl chambers, and we provide it later in this section.
The proof of part~(b) occupies Section~\ref{sec.proof of thm.main}.
Our second main result (Theorem~\ref{thm.asymptotic}) leverages Theorem~\ref{thm.main}(b) to compute $\kappa(G)$ for the standard representation $G$ of any irreducible Coxeter group, which in turn determines $\kappa(G)$ for any finite reflection group $G$ by Theorem~\ref{thm.main}(c).
The proof of Theorem~\ref{thm.asymptotic} appears in Section~\ref{sec.proof of thm.asymptotic}.

\begin{theorem}
\label{thm.asymptotic}
The optimal max filtering condition numbers of the exceptional Coxeter groups are reported in Table~\ref{table.condition numbers}.
For the remaining Coxeter groups, we have
\begin{itemize}
\item[(a)]
$\kappa(A_\ell)=(\frac{2}{\pi}+o(1))\ell$,
\item[(b)]
$\kappa(B_\ell)=(\frac{4}{\pi}+o(1))\ell$,
\item[(c)]
$\kappa(D_\ell)=(\frac{4}{\pi}+o(1))\ell$, and
\item[(d)]
$\kappa(I_m)=\sqrt{\frac{1+\cos(\frac{\pi}{m})}{1-\cos(\frac{\pi}{m})}}=(\frac{2}{\pi}+o(1))m$.
\end{itemize}
\end{theorem}

\begin{table}
\caption{\label{table.condition numbers}
Optimal max filtering condition numbers of exceptional Coxeter groups
}
\begin{center}
\begingroup
\setlength{\tabcolsep}{12pt}
\begin{tabular}{|crl|}\hline
$G$ & $\kappa(G)$ & minimal polynomial \\ \hline\hline
$E_6$ & $7.5957$ & $x^4 - 8 x^3 + 2 x^2 + 8 x + 1$\\ \hline
$E_7$ & $11.4301$ & $x^6 - 12x^5 + 3x^4+40x^3+3x^2-12x+1$\\ \hline
$E_8$ & $19.0811$ & $x^8-16x^7-60x^6+16x^5+134x^4+16x^3-60x^2-16x+1$ \\ \hline
$F_4$ & $7.5957$ & $x^4 - 8 x^3 + 2 x^2 + 8 x + 1$\\ \hline
$H_3$ & $6.3137$ & $x^4-4x^3-14x^2-4x+1$\\ \hline
$H_4$ & $19.0811$ & $x^8-16x^7-60x^6+16x^5+134x^4+16x^3-60x^2-16x+1$ \\ \hline
\end{tabular}
\endgroup
\end{center}
\end{table}

Before proving Theorem~\ref{thm.main}(c), we first motivate our approach with an example.
When $G$ is the group of $d\times d$ permutation matrices, it is straightforward to verify that
\begin{equation}
\label{eq.sort}
\llangle G\cdot x,G\cdot y\rrangle
=\langle \operatorname{sort}(x),\operatorname{sort}(y)\rangle.
\end{equation}
It turns out that every reflection group has a map analogous to sorting.
As motivation, we observe that the image of the sorting map is the polyhedral cone
\[
\{(x_1,\ldots,x_d)\in\mathbb{R}^d:x_1\geq\cdots\geq x_d\}.
\]
Every orbit $G\cdot x$ intersects this cone at a single point, namely, at $\operatorname{sort}(x)$.
One may show that the above cone is the closure of a Weyl chamber of $G$, which offers a hint for how to generalize the sorting map.

Given a finite reflection group $G\leq\operatorname{O}(d)$ and a corresponding Weyl chamber $C\subseteq\mathbb{R}^d$, we define $\operatorname{rep}\colon\mathbb{R}^d/G\to\overline{C}$ so that $\operatorname{rep}(G\cdot x)$ is the unique member of $(G\cdot x)\cap\overline{C}$; indeed, uniqueness is established on page~62 in~\cite{Kane:01}, which also demonstrates that $\operatorname{rep}$ is bijective.
In addition, $\operatorname{rep}$ preserves distances:

\begin{lemma}
\label{lem.isometry}
The bijection $\operatorname{rep}\colon \mathbb{R}^d/G\to\overline{C}$ defined above is an isometry.
\end{lemma}

This implies the following generalization of \eqref{eq.sort}:
\begin{align}
\llangle G\cdot x,G\cdot y\rrangle
\nonumber
&=\frac{1}{2}\Big(\|x\|^2+\|y\|^2-d(G\cdot x,G\cdot y)^2\Big)\\
\nonumber
&=\frac{1}{2}\Big(\|\operatorname{rep}(G\cdot x)\|^2+\|\operatorname{rep}(G\cdot y)\|^2-\|\operatorname{rep}(G\cdot x)-\operatorname{rep}(G\cdot y)\|^2\Big)\\
\label{eq.representative representation}
&=\langle \operatorname{rep}(G\cdot x),\operatorname{rep}(G\cdot y)\rangle.
\end{align}

\begin{proof}[Proof of Lemma~\ref{lem.isometry}]
A change of variables gives
\[
d(G\cdot x,G\cdot y)
=\min_{g,h\in G}\|gx-hy\|
=\min_{y'\in G\cdot y}\|\operatorname{rep}(G\cdot x)-y'\|.
\]
We note for future use that $(x,y)\mapsto d(G\cdot x,G\cdot y)$ is continuous since the first equality above expresses it as a minimum of finitely many continuous functions.

Suppose $x,y\in C$.
For each $y''\in G\cdot y$ with $y''\neq y$, there is a hyperplane orthogonal to some $\alpha$ in the fundamental system of $C$ that strictly separates $y''$ from $y\in C$, and so
\[
\|x-y''\|
>\|x-R(\alpha)y''\|
\geq \min_{y'\in G\cdot y}\|x-y'\|.
\]
It follows that $\|x-y'\|$ is minimized over $y'\in G\cdot y$ at $y'=y$, i.e.,
\[
d(G\cdot x,G\cdot y)
=\|x-y\|
=\|\operatorname{rep}(G\cdot x)-\operatorname{rep}(G\cdot y)\|.
\]

Now consider any $x,y\in\mathbb{R}^d$, and take sequences $\{x_n\}_{n=1}^\infty$ and $\{y_n\}_{n=1}^\infty$ in $C$ such that $x_n\to \operatorname{rep}(G\cdot x)$ and $y_n\to \operatorname{rep}(G\cdot y)$.
Then continuity implies
\[
d(G\cdot x_n,G\cdot y_n)
\to d\Big(G\cdot \operatorname{rep}(G\cdot x),G\cdot \operatorname{rep}(G\cdot y)\Big)
= d(G\cdot x,G\cdot y).
\]
Meanwhile, the previous argument gives
\[
d(G\cdot x_n,G\cdot y_n)
=\|x_n-y_n\|
\to\|\operatorname{rep}(G\cdot x)-\operatorname{rep}(G\cdot y)\|.
\]
It follows that $d(G\cdot x,G\cdot y)=\|\operatorname{rep}(G\cdot x)-\operatorname{rep}(G\cdot y)\|$, as desired.
\end{proof}

By virtue of \eqref{eq.representative representation}, we may represent max filtering as a linear map precomposed with the $\operatorname{rep}$ map.
While the image of $\operatorname{rep}$ is $C$, the linear map is defined on the full space.
As the following lemma shows, extending the linear map does not alter its condition number.

\begin{lemma}
\label{lem.linearize}
Let $G\leq\operatorname{O}(d)$ be a finite reflection group.
Given $T:=\{t_i\}_{i=1}^n\in(\mathbb{R}^d)^n$, define $\Phi\colon\mathbb{R}^d/G\to\mathbb{R}^n$ and $L\colon\mathbb{R}^d\to\mathbb{R}^n$ by
\[
\Phi(G\cdot x):=\{\llangle G\cdot t_i,G\cdot x\rrangle\}_{i=1}^n,
\qquad
L(x):=\{\langle \operatorname{rep}(G\cdot t_i),x\rangle\}_{i=1}^n.
\]
Then $\Phi$ and $L$ have identical upper and lower Lipschitz bounds.
\end{lemma}

\begin{proof}
By Lemma~\ref{lem.isometry}, $\operatorname{rep}\colon \mathbb{R}^d/G\to\overline{C}$ is a bijective isometry.
Letting $L_{\overline{C}}\colon \overline{C}\to\mathbb{R}^n$ denote the restriction of $L$ to $\overline{C}\subseteq\mathbb{R}^d$, we have $\Phi=L_{\overline{C}}\circ \operatorname{rep}$, and so it suffices to show that $L_{\overline{C}}$ and $L$ have identical upper and lower Lipschitz bounds.
Since $L$ is linear, its Lipschitz bounds are precisely its top and bottom singular values, and by restriction, it holds that
\[
\sigma_{\min}(L)
\leq \frac{\|L_{\overline{C}}(x)-L_{\overline{C}}(y)\|}{\|x-y\|}
\leq \sigma_{\max}(L)
\]
for every $x,y\in\overline{C}$ with $x\neq y$.
To show that these bounds are optimal for $L_{\overline{C}}$, select any $y\in C$, and let $z$ be a top (respectively, bottom) right-singular unit vector of $L$.
Since $C$ is open, it holds that $x:=y+\epsilon z\in C$ for an appropriately small $\epsilon>0$, in which case
\[
\frac{\|L_{\overline{C}}(x)-L_{\overline{C}}(y)\|}{\|x-y\|}
=\frac{\|L(x)-L(y)\|}{\|x-y\|}
=\frac{\|L(x-y)\|}{\|x-y\|}
=\frac{\|L(\epsilon z)\|}{\|\epsilon z\|}
=\|L z\|,
\]
which achieves equality in the upper (respectively, lower) Lipschitz bound, as desired.
\end{proof}

To find the best conditioned max filtering map, Lemma~\ref{lem.linearize} suggests finding templates $\{t_i\}_{i=1}^{n}$ for which the vectors $\{\operatorname{rep}(G\cdot t_i)\}_{i=1}^n$ form the rows of an optimally conditioned matrix.
Since these vectors are constrained to reside in the closed Weyl chamber $\overline{C}$, this optimization problem motivates the following definition.

\begin{definition}
The \textbf{optimal condition number} $\kappa(A)$ of a set $A\subseteq\mathbb{R}^d$ is the infimum of quotients $C/c$ for which there exists $n\in\mathbb{N}$ and vectors $\{t_i\}_{i=1}^n$ in $A$ such that
\[
c\|x\|\leq\bigg(\sum_{i=1}^n|\langle t_i,x\rangle|^2\bigg)^{1/2}\leq C\|x\|
\qquad
\forall x\in\mathbb{R}^d.
\]
Notably, $C/c$ is the condition number of the $d\times n$ matrix with columns $\{t_i\}_{i=1}^n$.
\end{definition}

The optimal condition number of a (possibly nonconvex) cone interacts nicely with direct sums in the sense that the worst distortion is the distortion of the worst behaved summand.
This is made precise in the following lemma.

\begin{lemma}
\label{lem.direct sum condition number}
Given spanning sets $X\subseteq\mathbb{R}^{d_X}$ and $Y\subseteq\mathbb{R}^{d_Y}$ that are closed under positive scalar multiplication, the direct sum $X\oplus Y\subseteq\mathbb{R}^{d_X+d_Y}$ has optimal condition number
\[
\kappa(X\oplus Y)=\max\{\kappa(X),\kappa(Y)\}.
\]
\end{lemma}

\begin{proof}
To further overload notation, given $A\in\mathbb{R}^{d\times n}$, we let $\kappa(A)$ denote the condition number of $A$.
It suffices to show two things:
\begin{itemize}
\item[(a)]
For all matrices $U$ and $V$ with columns from $X$ and $Y$, respectively, there exists a matrix $W$ with columns from $X\oplus Y$ such that $\kappa(W)\leq\max\{\kappa(U),\kappa(V)\}$.
\item[(b)]
For every matrix $W$ with columns from $X\oplus Y$, there exist matrices $U$ and $V$ with columns from $X$ and $Y$, respectively, such that $\max\{\kappa(U),\kappa(V)\}\leq\kappa(W)$.
\end{itemize}
For (a), we may assume $\kappa(U),\kappa(V)<\infty$.
Scale $U$ and $V$ so that $\sigma_{\min}(U)=\sigma_{\min}(V)=1$.
Taking $W$ to be block diagonal with blocks $U$ and $V$, then $\sigma_{\min}(W)=1$ and so
\[
\kappa(W)
=\max\{\sigma_{\max}(U),\sigma_{\max}(V)\}
=\max\{\kappa(U),\kappa(V)\}.
\]
For (b), we may assume $\kappa(W)<\infty$.
We may decompose $W$ as
\[
W=\left[\begin{array}{c}
U\\
V
\end{array}\right],
\]
where $U$ has columns from $X$ and $V$ has columns from $Y$.
We claim that 
\[
0
<\sigma_{\min}(W)
\leq\sigma_{\min}(U)
\leq\sigma_{\max}(U)
\leq\sigma_{\max}(W)
<\infty,
\]
and similarly for $V$; this in turn implies the desired inequality.
Indeed,
\[
\sigma_{\max}(W)
=\max_{\|x\|=1}\|W^\top x\|
\geq\max_{\|y\|=1} \left\|\left[\begin{array}{cc}
U^\top&V^\top
\end{array}\right]
\left[\begin{array}{c}y\\0\end{array}\right]\right\|
=\max_{\|y\|=1}\|U^\top y\|
=\sigma_{\max}(U),
\]
and a similar argument gives the desired bound for $\sigma_{\min}(U)$.
\end{proof}

We are now ready to prove Theorem~\ref{thm.main}(c).

\begin{proof}[Proof of Theorem~\ref{thm.main}(c)]
Combining Lemmas~\ref{lem.linearize} and~\ref{lem.direct sum condition number} gives $\kappa(G)=\max_i\kappa(G_i)$.
Next, we view $\mathbb{R}^d$ as a direct sum of $\mathbb{R}^{d_i}$ according to the factorization of $G$, and let $B_i$ denote the matrix whose columns are the minimal optimal templates for $\mathbb{R}^{d_i}$.
Then by Lemma~\ref{lem.spectral phenomenon}(b) and Theorem~\ref{thm.main}(a) and~(b),
\[
\frac{1}{\sigma_{\max}(B_i)^2}+\frac{1}{\sigma_{\min}(B_i)^2}=2,
\]
from which is follows that $\sigma_{\min}(B_i)$ and $\sigma_{\max}(B_i)$ are inversely related.
Take $B\in\mathbb{R}^{d\times d}$ to be block diagonal with blocks $B_i$.
Then for all $i$, it holds that
\[
\sigma_{\min}(B)
\leq\sigma_{\min}(B_i)
\leq\sigma_{\max}(B_i)
\leq\sigma_{\max}(B).
\]
Furthermore, since $\sigma_{\min}(B_i)$ and $\sigma_{\max}(B_i)$ are inversely related, there exists $j$ such that both $\sigma_{\min}(B_j)=\sigma_{\min}(B)$ and $\sigma_{\max}(B_j)=\sigma_{\max}(B)$.
It follows that
\[
\frac{\sigma_{\max}(B)}{\sigma_{\min}(B)}
=\frac{\sigma_{\max}(B_j)}{\sigma_{\min}(B_j)}
=\max_i\frac{\sigma_{\max}(B_i)}{\sigma_{\min}(B_i)}
=\max_i\kappa(G_i)
=\kappa(G),
\]
i.e., the columns of $B$ are optimal templates.
Finally, since a finite condition number requires $B$ to have at least $d$ columns, $B$ is minimal.
\end{proof}

\section{Proof of Theorem~\ref{thm.main}(b)}
\label{sec.proof of thm.main}

In this section, we leverage convex duality to establish Theorem~\ref{thm.main}(b).
Our reduction to convex programming is summarized by the following:

\begin{lemma}
\label{lem.convex programming reduction}
Let $G\leq\operatorname{O}(d)$ be a finite reflection group with Weyl chamber $C\subseteq\mathbb{R}^d$.
Denote
\[
\alpha:=\inf\Bigg\{\bigg\|\sum_{i=1}^n t_it_i^\top-I_d\bigg\|_{2\to2}:n\in\mathbb{N},t_1,\ldots,t_n\in\overline{C}\Bigg\}.
\]
Then $\kappa(G)=\sqrt{\frac{1+\alpha}{1-\alpha}}$.
\end{lemma}

\begin{proof}
By Lemma~\ref{lem.linearize}, every max filter bank shares upper and lower Lipschitz bounds with a linear map $L\colon\mathbb{R}^d\to\mathbb{R}^n$ of the form $L(x)=\{\langle t_i,x\rangle\}_{i=1}^n$ for some $t_1,\ldots,t_n\in\overline{C}$.
These bounds are the square roots of the top and bottom eigenvalues of $LL^*$, whose matrix representation is $\sum_{i=1}^nt_it_i^\top$.
The scaling of $L$ that makes the eigenvalues of $LL^*$ uniformly closest to $1$ makes the top and bottom eigenvalues exhibit a common distance $r$ from $1$.
Meanwhile, this scaling does not affect the quotient $\sqrt{\frac{1+r}{1-r}}$ of Lipschitz bounds.
The result follows.
\end{proof}

The definition of $\alpha$ in Lemma~\ref{lem.convex programming reduction} was directly inspired by Cahill and Chen's treatment of scalable frames in~\cite{CahillC:13}.
Scalable frames are concerned with best choices of vectors belonging to a particular union of $1$-dimensional cones (i.e., rays).
Our result replaces this union with a particular closed convex cone, namely, a Weyl chamber of a finite reflection group.

\textbf{For the remainder of this section, we assume that $G$ is an essential and irreducible finite reflection group, and we use the following notation:} 
Take $A\in\mathbb{R}^{d\times d}$ to be the matrix whose unit-norm rows form the fundamental system corresponding to the Weyl chamber $C$, and put $B:=A^{-1}$.
Then
\[
\overline{C}
=\{x\in\mathbb{R}^d:Ax\geq0\}
=\{By:y\geq0\}.
\]
Observe that for each $t\in\overline{C}$, there exists $y\geq0$ such that $tt^\top = Byy^\top B^\top$.
By factoring out $B$ and $B^\top$, arbitrary sums of such $tt^\top$ correspond to arbitrary sums of $yy^\top$, which in turn make up the convex cone $\operatorname{CP}(d)$ of $d\times d$ \textbf{completely positive matrices}.
Thus,
\[
\bigg\{\sum_{i=1}^n t_it_i^\top:n\in\mathbb{N},t_1,\ldots,t_n\in\overline{C}\bigg\}
=\Big\{BMB^\top:M\in\operatorname{CP}(d)\Big\}.
\]
In particular, the program in Lemma~\ref{lem.convex programming reduction} can be reformulated as
\begin{equation}
\label{eq.main primal}
\text{minimize}
\quad
\|BMB^\top-I_d\|_{2\to2}
\quad
\text{subject to}
\quad
M\in\operatorname{CP}(d).
\end{equation}
Our approach is to prove a sequence of technical conditions on the optimizers of \eqref{eq.main primal}.

\begin{lemma}
\label{lem.opt identity}
There exists an optimizer of \eqref{eq.main primal} that is a positive multiple of $I_d$.
\end{lemma}

Indeed, Lemma~\ref{lem.opt identity} allows us to prove our result:

\begin{proof}[Proof of Theorem~\ref{thm.main}(b)]
By Lemmas~\ref{lem.opt identity}, \ref{lem.convex programming reduction}, and~\ref{lem.spectral phenomenon}(c), we have
\[
\kappa(G)=\sqrt{\frac{1+\|cBB^\top-I_d\|_{2\to2}}{1-\|cBB^\top-I_d\|_{2\to2}}}
=\sqrt{\frac{\lambda_{\max}(BB^\top)}{\lambda_{\min}(BB^\top)}}
=\frac{\sigma_{\max}(B)}{\sigma_{\min}(B)}
=\frac{\sigma_{\max}(A)}{\sigma_{\min}(A)},
\]
with optimal templates given by the columns of $B$.
\end{proof}

To prove Lemma~\ref{lem.opt identity}, we first prove a weaker result:

\begin{lemma}
\label{lem.opt diag}
There exists an optimizer of \eqref{eq.main primal} that is diagonal.
\end{lemma}

We will prove Lemma~\ref{lem.opt diag} by observing an interplay between dual and restricted programs.
A restriction of the primal program (for instance to a diagonal primal variable) corresponds to a relaxation of the dual program. 
To show the restriction does not change the value of the primal program, we can show that an optimizer of the relaxed dual program is still feasible in the original dual program. 

The dual program of \eqref{eq.main primal} is naturally expressed in terms of the convex cone $\operatorname{COP}(d)$ of \textbf{copositive matrices}, that is, symmetric matrices $A\in\mathbb{R}^{d\times d}$ such that $x^\top Ax\geq0$ for all $x\geq0$:
\begin{equation}
\label{eq.main dual}
\text{maximize}
\quad
\operatorname{tr}W
\quad
\text{subject to}
\quad
\|W\|_*\leq 1,
\quad
-B^\top WB\in\operatorname{COP}(d).
\end{equation}
Consider the restriction of \eqref{eq.main primal} to diagonal $M\in\operatorname{CP}(d)$, i.e., $M=\operatorname{diag}(x)$ with $x\geq0$:
\begin{equation}
\label{eq.diagonal primal}
\text{minimize}
\quad
\|B\operatorname{diag}(x)B^\top-I_d\|_{2\to2}
\quad
\text{subject to}
\quad
x\geq0.
\end{equation}
The dual of this restriction is a relaxation of \eqref{eq.main dual}:
\begin{equation}
\label{eq.diagonal dual}
\text{maximize}
\quad
\operatorname{tr}Y
\quad
\text{subject to}
\quad
\|Y\|_*\leq 1,
\quad
Y^\top=Y,
\quad
\operatorname{diag}(B^\top YB)\leq0.
\end{equation}
Since the feasibility regions of \eqref{eq.main dual} and \eqref{eq.diagonal primal} each have nonempty interior, they satisfy strong duality by Slater's condition.
We will show that \eqref{eq.diagonal dual} has an optimizer that is feasible in \eqref{eq.main dual}.
By strong duality, this in turn implies Lemma~\ref{lem.opt diag}.
First, we use complementary slackness to derive useful statements about optimizers, for example:

\begin{lemma}\
\label{lem.opt in diag programs}
\begin{itemize}
\item[(a)]
Any optimal $x$ in \eqref{eq.diagonal primal} necessarily satisfies $\|B\operatorname{diag}(x)B^\top-I_d\|_{2\to2}<1$ and $x>0$.
\item[(b)]
Any optimal $Y$ in \eqref{eq.diagonal dual} necessarily satisfies $\operatorname{diag}(B^\top YB)=0$.
\end{itemize}
\end{lemma}

\begin{proof}
If $x$ has a zero entry, then $B\operatorname{diag}(x)B^\top$ is rank deficient, and so the objective in \eqref{eq.diagonal primal} at $x$ is at least $1$.
Meanwhile, taking $x:=\sigma_{\max}(B)^{-2}\cdot\mathbf{1}$ gives $0\prec B\operatorname{diag}(x)B^\top\preceq I_d$, and so the objective in \eqref{eq.diagonal primal} at $x$ is strictly less than $1$.
This implies (a), and then (b) follows from complementary slackness.
\end{proof}

We can say even more about the optimal $x$ and $Y$ by leveraging the following version of the von Neumann trace inequality:

\begin{proposition}[symmetric von Neumann trace inequality]
\label{prop.von neumann}
Suppose $A,B\in\mathbb{R}^{n\times n}$ are symmetric with eigenvalues $\lambda_1\leq\cdots\leq\lambda_n$ and $\mu_1\leq\cdots\leq\mu_n$, respectively.
Then
\[
\operatorname{tr}(AB)
\leq\sum_{i=1}^n\lambda_i\mu_i,
\]
with equality precisely when there exist orthonormal $\{u_i\}_{i=1}^n$ in $\mathbb{R}^n$ such that
\[
A=\sum_{i=1}^n\lambda_iu_iu_i^\top
\qquad
\text{and}
\qquad
B=\sum_{i=1}^n\mu_iu_iu_i^\top.
\]
\end{proposition}

\begin{proof}
First, $A$ and $B$ of the prescribed form achieve equality in the given bound.
It remains to prove the bound and then show that equality requires $A$ and $B$ to have this form.

Consider the case where $A$ and $B$ are both positive definite.
Then the eigenvalues coincide with singular values, and so Theorem~3.1 in~\cite{Carlsson:21} gives the bound, with equality precisely when there exist $V,W\in\operatorname{O}(n)$ such that $A=V\operatorname{diag}(\lambda)W^\top$ and $B=V\operatorname{diag}(\mu)W^\top$.
Since $A$ is positive definite, it has polar decomposition $A=(VW^\top)(W\operatorname{diag}(\lambda)W^\top)$, whose uniqueness implies $VW^\top=I_n$, i.e., $V=W$.
This resolves the positive definite case.

Now suppose $A$ and $B$ are symmetric, but not necessarily positive definite.
Then there exist $\alpha,\beta\in\mathbb{R}$ such that $A+\alpha I_n,B+\beta I_n\succ0$.
Our analysis above then gives
\begin{align*}
\operatorname{tr}(AB)
&=\operatorname{tr}((A+\alpha I_n)(B+\beta I_n))-\beta\operatorname{tr}A-\alpha\operatorname{tr}(B)-\alpha\beta n\\
&\leq\sum_{i=1}^n(\lambda_i+\alpha)(\mu_i+\beta)-\beta\sum_{i=1}^n\lambda_i-\alpha\sum_{i=1}^n\mu_i-\alpha\beta n
=\sum_{i=1}^n\lambda_i\mu_i,
\end{align*}
with equality precisely when there exist orthonormal $\{u_i\}_{i=1}^n$ in $\mathbb{R}^n$ such that
\[
A+\alpha I_n=\sum_{i=1}^n(\lambda_i+\alpha)u_iu_i^\top
\qquad
\text{and}
\qquad
B+\beta I_n=\sum_{i=1}^n(\mu_i+\beta)u_iu_i^\top.
\]
The result then follows since $\sum_iu_iu_i^\top=I_n$.
\end{proof}

Next, we show that for every optimal $x$ in \eqref{eq.diagonal primal} and $Y$ in \eqref{eq.diagonal dual}, the top eigenvector of $B\operatorname{diag}(x)B^\top$ is the bottom eigenvector of $Y$, and it resides in $C$.

\begin{lemma}\
\label{eq.eigenstructure of opt x and Y}
\begin{itemize}
\item[(a)]
For every optimal $x$ in \eqref{eq.diagonal primal}, the top eigenvalue of $B\operatorname{diag}(x)B^\top$ is simple, and the corresponding eigenspace is spanned by a positive combination of the columns of $B$.
\item[(b)]
Every optimal $Y$ in \eqref{eq.diagonal dual} has a unique negative eigenvalue, this eigenvalue is simple, and the corresponding eigenspace is identical to the top eigenspace in (a).
\end{itemize}
\end{lemma}

\begin{proof}
For (a), let $D$ denote the square root of $\operatorname{diag}(x)$, and put $S:=BD$.
Since $S$ is square, we have that $B\operatorname{diag}(x)B^\top=SS^\top$ has the same spectrum as $S^\top S=DB^\top BD$ with the same multiplicities.
By Lemma~\ref{lem.acute cone}, $B^\top B$ is entrywise positive.
Since $x>0$ by Lemma~\ref{lem.opt in diag programs}(a), it holds that $S^\top S$ is also entrywise positive, and so Perron--Frobenius gives that the top eigenvalue is simple with an entrywise positive eigenvector $v$.
Then the top eigenvector of $B\operatorname{diag}(x)B^\top=SS^\top$ is $Sv=BDv$, i.e., a positive combination of the columns of $B$.

For (b), we first give a proof of weak duality between \eqref{eq.diagonal primal} and \eqref{eq.diagonal dual}:
For every feasible $x$ in \eqref{eq.diagonal primal} and $Y$ in \eqref{eq.diagonal dual}, if we let $\lambda,\mu\in\mathbb{R}^d$ denote the vectors of sorted eigenvalues of $Y$ and $I_d-B\operatorname{diag}(x)B^\top$, respectively, then symmetric von Neumann and H\"{o}lder together give
\begin{align}
\operatorname{tr}Y
\nonumber
&\leq\operatorname{tr}Y - \langle B^\top YB,\operatorname{diag}(x)\rangle\\
\nonumber
&=\langle Y,I_d-B\operatorname{diag}(x)B^\top\rangle\\
\label{eq.weak duality between diag programs}
&\leq\langle \lambda,\mu\rangle
\leq\|\lambda\|_1\|\mu\|_\infty
=\|Y\|_*\|B\operatorname{diag}(x)B^\top-I_d\|_{2\to2}
\leq\|B\operatorname{diag}(x)B^\top-I_d\|_{2\to2}.
\end{align}
If $x$ is optimal, then we may combine a portion of \eqref{eq.weak duality between diag programs} with Lemma~\ref{lem.opt in diag programs}(a) to get
\[
\operatorname{tr}Y
\leq\|Y\|_*\|B\operatorname{diag}(x)B^\top-I_d\|_{2\to2}
<\|Y\|_*.
\]
Thus, every feasible $Y$ has a strictly negative eigenvalue, i.e., $\lambda_1<0$.
Next, any optimal $x$ and $Y$ together achieve equality in the following portion of \eqref{eq.weak duality between diag programs}:
\[
\langle \lambda,\mu\rangle
\leq\|\lambda\|_1\|\mu\|_\infty.
\]
Equality requires $\operatorname{supp}(\lambda)\subseteq\arg\max_i|\mu_i|$ and $\lambda_i\mu_i\geq0$ for all $i$.
Since $\lambda_1<0$, we have $1\in\arg\max_i|\mu_i|$.
As a consequence of the simplicity conclusion of part~(a), $\mu$ is nonzero, meaning $|\mu_1|=\max_i|\mu_i|>0$, which, combined with $\lambda_1<0$ and $\lambda_1\mu_1\geq0$, implies $\mu_1<0$.
Next, part~(a) gives $\mu_1<\mu_i$ for all $i>1$.
It follows that any $j\in\arg\max_i|\mu_i|$ with $j\neq 1$ must have $\mu_j>0$, and so $\lambda_j\geq0$.
As such, $\lambda_1$ is the only negative eigenvalue of $Y$.

Finally, any optimal $x$ and $Y$ achieve equality in the following portion of \eqref{eq.weak duality between diag programs}:
\[
\langle Y,I_d-B\operatorname{diag}(x)B^\top\rangle
\leq\langle \lambda,\mu\rangle.
\]
By part~(a), we have $\mu_1<\mu_2$, while the above gives $\lambda_1<\lambda_2$.
As such, Proposition~\ref{prop.von neumann} gives that the bottom eigenspace of $Y$ is identical to the top eigenspace of $B\operatorname{diag}(x)B^\top$.
\end{proof}

\begin{proof}[Proof of Lemma~\ref{lem.opt diag}]
Given an optimizer $x$ of \eqref{eq.diagonal primal}, we will show that $M:=\operatorname{diag}(x)$ is an optimizer of \eqref{eq.main primal}.
Such a choice of $M$ is feasible in \eqref{eq.main primal}.
By strong duality, it suffices to show that any optimal $Y$ in \eqref{eq.diagonal dual} is feasible in \eqref{eq.main dual}.
To this end, by Lemma~\ref{eq.eigenstructure of opt x and Y}, we may diagonalize $Y=UDU^\top$ so that the first column of $U$ is a positive combination of the columns of $B$ and $D=\operatorname{diag}(\lambda_1,\ldots,\lambda_d)$ with $\lambda_1<0\leq\lambda_2\leq\cdots\leq\lambda_d$.
Consider the ellipsoid
\[
E
:=\{z\in\mathbb{R}^d:z^\top Dz\leq0,z_1=1\}
=\{z\in\mathbb{R}^d:\lambda_2z_2^2+\cdots+\lambda_dz_d^2\leq|\lambda_1|,z_1=1\}.
\]
Notably, $E$ is convex.
Fix $i\in[d]$ and let $b_i$ denote the $i$th column of $B$.
We claim that $v:=U^\top b_i\in\operatorname{cone}(E)$.
By Lemma~\ref{lem.opt in diag programs}(b), we have
\[
-|\lambda_1|v_1^2+\lambda_2v_2^2+\cdots+\lambda_dv_d^2
=v^\top Dv
=b_i^\top Yb_i
=(B^\top YB)_{ii}
=0.
\]
Rearranging then gives
\[
\lambda_2v_2^2+\cdots+\lambda_dv_d^2
=|\lambda_1|v_1^2.
\]
Lemmas~\ref{eq.eigenstructure of opt x and Y} and~\ref{lem.acute cone} together give $v_1>0$, and so $v\in\operatorname{cone}(E)$, as claimed.
Since our choice for $i\in[d]$ was arbitrary, it follows that the closed chamber $\overline{C}$ (which is generated by the columns of $B$) satisfies $U^\top\overline{C}\subseteq\operatorname{cone}(E)$.
To show that $Y$ is feasible in \eqref{eq.main dual}, i.e., $-B^\top YB\in\operatorname{COP}(d)$, take any vector $x\geq0$.
Then $Bx\in\overline{C}$, and so $v:=U^\top Bx\in\operatorname{cone}(E)$, meaning
\[
x^\top B^\top YBx
= x^\top B^\top UDU^\top Bx
= v^\top Dv
\leq 0,
\]
as desired.
\end{proof}

\begin{proof}[Proof of Lemma~\ref{lem.opt identity}]
By Lemma~\ref{lem.opt diag}, there exists an optimizer of \eqref{eq.main primal} that is diagonal, i.e., any optimizer $x$ of the restriction \eqref{eq.diagonal primal} corresponds to an optimizer $\operatorname{diag}(x)$ of \eqref{eq.main primal}.
Consider a further restriction of \eqref{eq.diagonal primal}:
\begin{equation}
\label{eq.scalar primal}
\text{minimize}
\quad
\|cBB^\top-I_d\|_{2\to2}
\quad
\text{subject to}
\quad
c\geq0.
\end{equation}
In what follows, we show that any optimizer $c$ of \eqref{eq.scalar primal} is strictly positive and corresponds to an optimizer $c\mathbf{1}$ of \eqref{eq.diagonal primal}.

First, the optimal $c$ in \eqref{eq.scalar primal} is strictly positive, since $c=0$ has value $1$, while $c=\sigma_{\max}(B)^{-2}$ has value strictly less than $1$ since the columns of $B$ span.
Next, the dual program of \eqref{eq.scalar primal} is\begin{equation}
\label{eq.scalar dual}
\text{maximize}
\quad
\operatorname{tr}Z
\quad
\text{subject to}
\quad
\|Z\|_*\leq1,
\quad
Z^\top=Z,
\quad
\operatorname{tr}(B^\top ZB)\leq 0.
\end{equation}
Since \eqref{eq.scalar primal} satisfies Slater's condition, we have strong duality, and so it suffices to show that an optimizer of the dual program \eqref{eq.scalar dual} is feasible in \eqref{eq.diagonal dual}.
To this end, we first give a proof of weak duality between \eqref{eq.scalar primal} and \eqref{eq.scalar dual}:
For every feasible $c$ in \eqref{eq.scalar primal} and $Z$ in \eqref{eq.scalar dual}, if we let $\lambda,\mu\in\mathbb{R}^d$ denote the vectors of sorted eigenvalues of $Z$ and $I_d-cBB^\top$, respectively, then symmetric von Neumann and H\"{o}lder together give
\begin{align*}
\operatorname{tr}Z
&\leq\operatorname{tr}Z-c\operatorname{tr}(B^\top ZB)\\
&=\langle Z,I_d-cBB^\top\rangle\\
&\leq\langle\lambda,\mu\rangle
\leq\|\lambda\|_1\|\mu\|_\infty
=\|Z\|_*\|cBB^\top-I_d\|_{2\to2}
\leq\|cBB^\top-I_d\|_{2\to2}.
\end{align*}
Equality in $\operatorname{tr}Z\leq\operatorname{tr}Z-c\operatorname{tr}(B^\top ZB)$ requires $\operatorname{tr}(B^\top ZB)=0$.
Next, by Proposition~\ref{prop.von neumann}, equality in $\langle Z,I_d-cBB^\top\rangle\leq\langle\lambda,\mu\rangle$ implies that $Z$ and $I_d-cBB^\top$ are simultaneously diagonalizable.
Equality in $\langle\lambda,\mu\rangle\leq\|\lambda\|_1\|\mu\|_\infty$ then gives that $\lambda_i\neq0$ only if $i\in\{1,d\}$.

Select any optimal $c$ in \eqref{eq.scalar primal} and optimal $Z$ in \eqref{eq.scalar dual}.
Let $\{u_i\}_{i=1}^d$ denote the orthonormal basis of eigenvectors of $I_d-cBB^\top$ afforded by Proposition~\ref{prop.von neumann}.
These eigenvectors appear in the singular value decomposition
\[
B=\sum_{i=1}^d \sigma_iu_iv_i^\top
\]
as well as the eigenvalue decomposition
\[
Z=\lambda_1u_1u_1^\top+\lambda_du_du_d^\top.
\]
Then $B^\top ZB=\sigma_1^2\lambda_1 v_1v_1^\top+\sigma_d^2\lambda_d v_dv_d^\top$, the trace of which is $\sigma_1^2\lambda_1+\sigma_d^2\lambda_d=0$ by the above complementary slackness argument.
Thus, $\operatorname{diag}(B^\top ZB)=0$ precisely when $(v_1)_i^2=(v_d)_i^2$ for all $i\in[d]$, which in turn holds by Lemma~\ref{lem.spectral phenomenon}(a).
As such, an optimizer of the dual program~\eqref{eq.scalar dual} is feasible in \eqref{eq.diagonal dual}, as desired.
\end{proof}

\section{Proof of Theorem~\ref{thm.asymptotic}}
\label{sec.proof of thm.asymptotic}

Theorem~\ref{thm.main} reports the best choice of max filtering templates in terms of the matrix whose unit-norm rows form a fundamental system of $G$.
Furthermore, the singular values of this matrix determine the optimal condition number.
Recalling Coxeter's classification, we compute these optimal condition numbers in Mathematica\footnote{\url{https://github.com/Daniel-Packer/reflection-groups/blob/main/ExceptionalGroups.nb}} for finitely many reflection groups.
The optimal condition numbers for the dihedral groups are easy to compute by hand:
\[
AA^\top
=\left[\begin{array}{cc}1&-\cos(\frac{\pi}{m})\\-\cos(\frac{\pi}{m})&1\end{array}\right]
\]
has eigenvalues $1\pm\cos(\frac{\pi}{m})$, and so $\kappa(I_m)=\sqrt{\frac{1+\cos(\frac{\pi}{m})}{1-\cos(\frac{\pi}{m})}}$, which is $(\frac{2}{\pi}+o(1))m$ by the Taylor series expansion of cosine.

The remaining families $A_\ell$, $B_\ell$, and $D_\ell$ correspond to $\ell\times\ell$ matrices whose characteristic polynomials are cumbersome to interact with as $\ell\to\infty$.
Instead, we determine the asymptotic form of the optimal condition numbers in these cases.
To do so, we embed the relevant matrices as integral operators over a common Hilbert space, where we can characterize spectral convergence as the matrix size grows.

\begin{definition}
The \textbf{pixel embedding} of $A\in\mathbb{R}^{n\times n}$ is $A^\sharp\in L^2([0,1]^2)$ defined by
\[
A^\sharp(x,y)
=\sum_{i=1}^n\sum_{j=1}^n A_{ij} \cdot n\cdot 1_{(\frac{i-1}{n},\frac{i}{n})}(x)\cdot 1_{(\frac{j-1}{n},\frac{j}{n})}(y),
\]
i.e., we chop the unit square into squares of width $1/n$, and then we take $A^\sharp$ to be constant on each square according to appropriately scaled entries of $A$.
We interpret $A^\sharp$ as the kernel of an integral operator $L^2([0,1])\to L^2([0,1])$, which we also denote $A^\sharp$ by an abuse of notation:
\[
(A^\sharp f)(x):=\int_0^1 A^\sharp(x,y)f(y)dy.
\]
\end{definition}

\begin{lemma}
\label{lem.limit norm}
Given a sequence $\{A_n\}_{n=1}^\infty$ of matrices of possibly different sizes such that $A_n^\sharp\to K$ in $L^2([0,1]^2)$, it holds that $\|A_n\|_{2\to2}\to\|K\|_{L^2\to L^2}$.
\end{lemma}

\begin{proof}
First, we observe some relationships between $A\in\mathbb{R}^{n\times n}$ and its pixel embedding $A^\sharp$:
\begin{itemize}
\item
The operator $A^\sharp$ acts invariantly on the subspace $P_n\leq L^2([0,1])$ spanned by the characteristic functions of intervals $(\frac{i-1}{n},\frac{i}{n})$.
\item
$A$ is the matrix representation of the restriction $A^\sharp|_{P_n}$ with respect to this basis.
\item
The kernel of $A^\sharp$ contains the orthogonal complement of $P_n$ in $L^2([0,1])$.
\end{itemize}
Thus, $\|A\|_{2\to2}=\|A^\sharp\|_{L^2\to L^2}$.
The result then follows from Cauchy--Schwarz:
\[
\big|\|A_n^\sharp\|_{L^2\to L^2}-\|K\|_{L^2\to L^2}\big|
\leq\|A_n^\sharp-K\|_{L^2\to L^2}
=\sup_{\|f\|_{L^2}=1}\|(A_n^\sharp-K)f\|_{L^2}
\leq\|A_n^\sharp-K\|_{L^2([0,1]^2)}.
\qedhere
\]
\end{proof}

In the following subsections, we apply Lemma~\ref{lem.limit norm} to treat $A_\ell$, $B_\ell$, and $D_\ell$.
Unlike $B_\ell$ and $D_\ell$, the action of $A_\ell$ is more naturally represented in an $(\ell+1)$-dimensional space.
For this reason, our analysis of $A_\ell$ is more intricate, while our analysis of $B_\ell$ and $D_\ell$ is straightforward.

\subsection{$A_\ell$}

Fix $\ell\in\mathbb{N}$, put $d=\ell$, and select $A\in\mathbb{R}^{d\times d}$ with unit-norm rows that form a fundamental system for $A_\ell$.
In this subsection, we estimate
\[
\kappa(A_\ell)
=\frac{\sigma_{\max}(A)}{\sigma_{\min}(A)}
=\sqrt{\frac{\lambda_{\max}(AA^\top)}{\lambda_{\min}(AA^\top)}}
=\sqrt{\frac{2-\lambda_{\min}(AA^\top)}{\lambda_{\min}(AA^\top)}},
\]
where the last equality applies Lemma~\ref{lem.spectral phenomenon}(b).
It will be convenient to analyze
\[
M:=\left[\begin{array}{cc}
AA^\top&0\\
0^\top &1
\end{array}\right],
\]
which satisfies $\lambda_{\min}(M)=\lambda_{\min}(AA^\top)$ by Lemma~\ref{lem.spectral phenomenon}(b).
From the Coxeter diagram of $A_\ell$, we have that $AA^\top$ is tridiagonal with $1$'s on the diagonal and $-\frac{1}{2}$'s above and below the diagonal.
Define $S\in\mathbb{R}^{(d+1)\times(d+1)}$ by
\[
S:=\left[\begin{array}{rrr|c}
\frac{1}{\sqrt{2}}&&&\frac{1}{\sqrt{d+1}}\\
-\frac{1}{\sqrt{2}}&\ddots&&\frac{1}{\sqrt{d+1}}\\
&\ddots&\frac{1}{\sqrt{2}}&\vdots\\
&&-\frac{1}{\sqrt{2}}&\frac{1}{\sqrt{d+1}}\\
\end{array}\right],
\]
where the blank entries are zeros.
Then $S^\top S=M$, and furthermore, one may verify that
\[
(S^{-1})_{ij}
=\left\{\begin{array}{cl}
\sqrt{2}\cdot(1-\frac{i}{d+1})&\text{if }j\leq i\leq d\\
\frac{1}{\sqrt{d+1}}&\text{if }j\leq i=d+1\\
-\sqrt{2}\cdot\frac{i}{d+1}&\text{if }j>i.
\end{array}\right.
\]
We wish to estimate
\[
\sqrt{\lambda_{\min}(AA^\top)}
=\sqrt{\lambda_{\min}(M)}
=\sigma_{\min}(S)
=\sigma_{\max}(S^{-1})^{-1}
=\tfrac{1}{d+1}\cdot\|\tfrac{1}{d+1}S^{-1}\|_{2\to2}^{-1}.
\]
It is straightforward to verify that as $d\to\infty$, the pixel embedding $(\frac{1}{d+1}S^{-1})^\sharp$ converges in $L^2([0,1]^2)$ to the kernel $K$ defined by
\[
K(x,y):=\left\{\begin{array}{cl}
\sqrt{2}\cdot(1-x)&\text{if }y\leq x\\
-\sqrt{2}\cdot x&\text{if }y>x.
\end{array}\right.
\]
Lemma~\ref{lem.kernel norm} below gives that $\|K\|_{L^2\to L^2}=\sqrt{2}/\pi$, which combined with Lemma~\ref{lem.limit norm} implies
\[
(d+1)\cdot\sqrt{\lambda_{\min}(AA^\top)}
=\|\tfrac{1}{d+1}S^{-1}\|_{2\to2}^{-1}
\to\|K\|_{L^2\to L^2}^{-1}
=\tfrac{\pi}{\sqrt{2}},
\]
and so recalling $d=\ell$, we have
\[
\kappa(A_\ell)
=\sqrt{\frac{2-\lambda_{\min}(AA^\top)}{\lambda_{\min}(AA^\top)}}
=(\tfrac{2}{\pi}+o(1)) \ell,
\]
as desired.

\begin{lemma}
\label{lem.kernel norm}
For $K$ defined above, it holds that $\|K\|_{L^2\to L^2}=\sqrt{2}/\pi$.
\end{lemma}

\begin{proof}
We start by identifying all eigenvalues and eigenvectors of $KK^*$.
To this end, a straightforward calculation gives
\[
(KK^*f)(x)
=2\int_0^x\int_y^1 f(z)dzdy-2x\int_0^1\int_y^1 f(z)dzdy.
\]
Observe that $(KK^*f)(0)=0=(KK^*f)(1)$ for any $f\in L^2([0,1])$, and so any eigenvector $g$ of $KK^*$ with nonzero eigenvalue must satisfy the boundary conditions
\[
g(0)=0=g(1).
\]
Next, two applications of the fundamental theorem of calculus gives $(KK^*f)''=-2f$, meaning $g$ is an eigenvector with eigenvalue $\lambda$ only if $\lambda g''=-2g$.
Notably, this implies that $0$ is not an eigenvalue of $KK^*$.
Since $KK^*$ is positive definite, it follows that every eigenvalue is strictly positive.
As such, $g$ is an eigenvector with eigenvalue $\lambda>0$ only if $g''=-\omega^2 g$ with $\omega:=\sqrt{2/\lambda}$, i.e., there exist $a,b\in\mathbb{R}$ such that
\[
g(x)=a\cos(\omega x)+b\sin(\omega x).
\]
The boundary condition $g(0)=0$ then implies that $a=0$, and since $g$ is an eigenvector, it follows that $b\neq0$.
The boundary condition $g(1)=0$ then gives $b\sin(\omega)=0$, i.e., $\omega=k\pi$ for some $k\in\mathbb{N}$.
Thus, every eigenvalue of $KK^*$ takes the form $2/(k\pi)^2$ for some $k\in\mathbb{N}$.
Furthermore, an easy calculation confirms that $g(x):=\sin(\pi x)$ satisfies $KK^*g=(2/\pi^2)g$.
The $C^*$ identity and the spectral theorem for compact self-adjoint operators then gives $\|K\|_{L^2\to L^2}^2=\lambda_{\max}(KK^*)=2/\pi^2$.
\end{proof}

\subsection{$B_\ell$ and $D_\ell$}

First, we address $B_\ell$.
Fix $\ell\in\mathbb{N}$, put $d=\ell$, and consider $A,B\in\mathbb{R}^{d\times d}$ defined by
\[
A:=
\left[\begin{array}{cccc}
\frac{1}{\sqrt{2}}&-\frac{1}{\sqrt{2}}&\phantom\ddots&\\
&\ddots&\ddots&\\
&\phantom\ddots&\frac{1}{\sqrt{2}}&-\frac{1}{\sqrt{2}}\\
&&\phantom\ddots&1
\end{array}\right],
\qquad
B:=
\left[\begin{array}{cccc}
\sqrt{2}&\cdots&\sqrt{2}&1\\
&\ddots&\vdots&\vdots\\
&&\sqrt{2}&1\\
&&&1
\end{array}\right].
\]
The rows of $A$ form a fundamental system for $B_\ell$, and $B=A^{-1}$.
As $d\to\infty$, the pixel embedding $(\frac{1}{d}B)^\sharp$ converges in $L^2([0,1]^2)$ to $K$ defined by
\[
K(x,y):=\left\{\begin{array}{cl}
\sqrt{2}&\text{if }y\geq x\\
0&\text{if }y<x.
\end{array}\right.
\]
The operator $\frac{1}{\sqrt{2}}K$ is the adjoint of the Volterra operator, whose operator norm is well known to be $2/\pi$ (by an argument similar to our analysis of $A_\ell$).
Then
\[
\lambda_{\min}(AA^\top)
=\lambda_{\max}(B^\top B)^{-1}
=\tfrac{1}{d^2}\cdot\|\tfrac{1}{d}B\|_{2\to2}^{-2}
=\tfrac{1}{d^2}\Big(\|K\|_{2\to2}^{-2}+o(1)\Big)
=\tfrac{1}{d^2}(\tfrac{\pi^2}{8}+o(1)),
\]
and so recalling $d=\ell$, we have
\[
\kappa(B_\ell)
=\sqrt{\frac{2-\lambda_{\min}(AA^\top)}{\lambda_{\min}(AA^\top)}}
=\sqrt{\frac{2-\tfrac{1}{d^2}(\tfrac{\pi^2}{8}+o(1))}{\tfrac{1}{d^2}(\tfrac{\pi^2}{8}+o(1))}}
=(\tfrac{4}{\pi}+o(1))\ell.
\]
To analyze $D_\ell$, we instead take
\[
A:=
\left[\begin{array}{cccr}
\frac{1}{\sqrt{2}}&-\frac{1}{\sqrt{2}}&\phantom\ddots&\\
&\ddots&\ddots&\\
&\phantom\ddots&\frac{1}{\sqrt{2}}&-\frac{1}{\sqrt{2}}\\
&\phantom\ddots&\frac{1}{\sqrt{2}}&\frac{1}{\sqrt{2}}
\end{array}\right],
\qquad
B:=
\left[\begin{array}{cccrc}
\sqrt{2}&\cdots&\sqrt{2}&\frac{1}{\sqrt{2}}&\frac{1}{\sqrt{2}}\\
&\ddots&\vdots&\vdots\phantom{~}&\vdots\\
&&\sqrt{2}&\frac{1}{\sqrt{2}}&\frac{1}{\sqrt{2}}\\
&&\phantom\vdots&-\frac{1}{\sqrt{2}}&\frac{1}{\sqrt{2}}
\end{array}\right].
\]
The rows of $A$ form a fundamental system for $D_\ell$, and $B=A^{-1}$.
As $d\to\infty$, the pixel embedding $(\frac{1}{d}B)^\sharp$ converges in $L^2([0,1]^2)$ to the same kernel $K$ defined above, and so the identical argument gives $\kappa(D_\ell)=(\tfrac{4}{\pi}+o(1))\ell$.

\section*{Acknowledgments}

The authors thank Jameson Cahill, Joey Iverson, John Jasper, and Yousef Qaddura for enlightening discussions.

\appendix

\section{Proof of Lemma~\ref{lem.spectral phenomenon}}
\label{appendix.spectral phenomenon}

\subsection{Proof of Lemma~\ref{lem.spectral phenomenon}(b) and~(c)}

The exceptional cases are treated in the Mathematica same notebook from Section~\ref{sec.proof of thm.asymptotic}.
In the case $G=I_m$, the eigenvalues of $M=-\cos(\frac{\pi}{m})\cdot[\begin{smallmatrix}0&1\\1&0\end{smallmatrix}]$ are $\pm\cos(\frac{\pi}{m})$, while $M^{(1)}=M^{(2)}=0$.
For the remaining cases where $G\in\{A_\ell,B_\ell,D_\ell\}$, we show that the characteristic polynomials of $M$ and $M^{(j)}$ are either even or odd.
Due to the form of these matrices, it is more convenient to instead work with $2M$ and $2M^{(j)}$.

First, suppose $G=A_\ell$ for $\ell\in\mathbb{N}$ and denote
\[
M(A_\ell)
:=M,
\qquad
P_\ell(\lambda)
:=\operatorname{det}(\lambda I_\ell-2M(A_\ell)).
\]
If $\ell=1$, then $M(A_\ell)=0$ and $P_1(\lambda)=\lambda$ is odd.
If $\ell=2$, then $G=I_3$ and $P_2(\lambda)=\lambda^2-1$ is even.
Otherwise, we have
\begin{equation}
\label{eq.matrix recursion}
\lambda I_\ell-2M(A_{\ell})=\left[\begin{array}{ccc}
\lambda&1&0_{\ell-2}^\top\\
1&\lambda&e_1^\top\\
0_{\ell-2}&e_1&\lambda I_{\ell-2}-2M(A_{\ell-2})
\end{array}\right],
\end{equation}
and so the cofactor expansion across the first row gives
\begin{align}
P_\ell(\lambda)
\nonumber
&=\lambda\operatorname{det}\left[\begin{array}{cc}
\lambda&e_1^\top\\
e_1&\lambda I_{\ell-2}-2M(A_{\ell-2})
\end{array}\right]
-\operatorname{det}\left[\begin{array}{cc}
1&e_1^\top\\
0_{\ell-2}&\lambda I_{\ell-2}-2M(A_{\ell-2})
\end{array}\right]\\
\label{eq.poly recursion}
&=\lambda P_{\ell-1}(\lambda)-P_{\ell-2}(\lambda).
\end{align}
By induction, it follows that $P_\ell(\lambda)$ is odd (even) precisely when $\ell$ is odd (even).
This proves Lemma~\ref{lem.spectral phenomenon}(b) in this case.
Meanwhile, Lemma~\ref{lem.spectral phenomenon}(c) in this case follows from the fact that
\[
\lambda I_{\ell-1}-2M(A_{\ell})^{(j)}=\left[\begin{array}{cc}
\lambda I_{j-1}-2M(A_{j-1})&0\\
0&\lambda I_{\ell-j}-2M(A_{\ell-j})
\end{array}\right],
\]
whose determinant is a product of odd and/or even functions of $\lambda$.

Next, suppose $G=B_\ell$ for $\ell\in\mathbb{N}\setminus\{1\}$ and denote
\[
M(B_\ell)
:=M,
\qquad
Q_\ell(\lambda)
:=\operatorname{det}(\lambda I_\ell-2M(B_\ell)).
\]
If $\ell=2$, then $G=I_4$ and $Q_2(\lambda)=\lambda^2-2$ is even.
If $\ell=3$, then
\[
Q_3(\lambda)
=\operatorname{det}\left[\begin{array}{ccc}
\lambda&1&0\\
1&\lambda&\sqrt{2}\\
0&\sqrt{2}&\lambda
\end{array}\right]
=\lambda^3-3\lambda
\]
is odd.
Otherwise, the same recursions \eqref{eq.matrix recursion} and \eqref{eq.poly recursion} hold for $M(B_\ell)$ and $Q_\ell$.
By induction, it follows that $Q_\ell(\lambda)$ is odd (even) precisely when $\ell$ is odd (even).
This proves Lemma~\ref{lem.spectral phenomenon}(b) in this case.
Meanwhile, denoting $X_1:=M(A_1)$ and otherwise $X_\ell:=M(B_\ell)$, Lemma~\ref{lem.spectral phenomenon}(c) in this case follows from the fact that
\[
\lambda I_{\ell-1}-2M(B_{\ell})^{(j)}=\left[\begin{array}{cc}
\lambda I_{j-1}-2M(A_{j-1})&0\\
0&\lambda I_{\ell-j}-2X_{\ell-j}
\end{array}\right],
\]
whose determinant is a product of odd and/or even functions of $\lambda$.

Finally, suppose $G=D_\ell$ for $\ell\in\mathbb{N}\setminus\{1,2\}$ and denote
\[
M(D_\ell)
:=M,
\qquad
R_\ell(\lambda)
:=\operatorname{det}(\lambda I_\ell-2M(D_\ell)).
\]
If $\ell=3$, then $G=A_3$ and so \eqref{eq.poly recursion} gives that $R_3(\lambda)=P_3(\lambda)$ is odd.
If $\ell=4$, then
\[
R_4(\lambda)
=\operatorname{det}\left[\begin{array}{cccc}
\lambda&1&0&0\\
1&\lambda&1&1\\
0&1&\lambda&0\\
0&1&0&\lambda
\end{array}\right]
=\lambda^4-3\lambda^2
\]
is even.
Otherwise, the same recursions \eqref{eq.matrix recursion} and \eqref{eq.poly recursion} hold for $M(D_\ell)$ and $R_\ell$.
By induction, it follows that $R_\ell(\lambda)$ is odd (even) precisely when $\ell$ is odd (even).
This proves Lemma~\ref{lem.spectral phenomenon}(b) in this case.
Meanwhile, denoting $Y_1:=M(A_1)$, $Y_2:=0\in\mathbb{R}^{2\times 2}$, and otherwise $Y_\ell:=M(D_\ell)$, Lemma~\ref{lem.spectral phenomenon}(c) in this case follows from the fact that
\[
\lambda I_{\ell-1}-2M(D_{\ell})^{(j)}=\left[\begin{array}{cc}
\lambda I_{j-1}-2M(A_{j-1})&0\\
0&\lambda I_{\ell-j}-2Y_{\ell-j}
\end{array}\right],
\]
whose determinant is a product of odd and/or even functions of $\lambda$.

\subsection{Proof of Lemma~\ref{lem.spectral phenomenon}(a)}

First, $AA^\top$ has the same eigenvectors as $M:=AA^\top-I_d$, and so we can combine Lemma~\ref{lem.spectral phenomenon}(b) and~(c) with the eigenvector--eigenvalue identity described in~\cite{DentonPTZ:21}:
\[
(v_1)_j^2
=\frac{\displaystyle\prod_{k=1}^{d-1}\Big|\lambda_1(M)-\lambda_k(M^{(j)})\Big|}{\displaystyle\prod_{k=2}^d\Big|\lambda_1(M)-\lambda_k(M)\Big|},
\qquad
(v_d)_j^2
=\frac{\displaystyle\prod_{k=1}^{d-1}\Big|\lambda_d(M)-\lambda_k(M^{(j)})\Big|}{\displaystyle\prod_{k=1}^{d-1}\Big|\lambda_d(M)-\lambda_k(M)\Big|}.
\]
The usual limiting argument allows us to assume that the denominators above are nonzero without loss of generality.
By Lemma~\ref{lem.spectral phenomenon}(b) and~(c), the numerators are equal (as are the denominators):
\[
\prod_{k=1}^{d-1}\Big|\lambda_1(M)-\lambda_k(M^{(j)})\Big|
=\prod_{k=1}^{d-1}\Big|-\lambda_d(M)+\lambda_{d-k}(M^{(j)})\Big|
=\prod_{k=1}^{d-1}\Big|\lambda_d(M)-\lambda_{k}(M^{(j)})\Big|,
\]
\[
\prod_{k=2}^d\Big|\lambda_1(M)-\lambda_k(M)\Big|
=\prod_{k=2}^d\Big|-\lambda_d(M)+\lambda_{d-k+1}(M)\Big|
=\prod_{k=1}^{d-1}\Big|\lambda_d(M)-\lambda_{k}(M)\Big|,
\]
where in both cases the last step follows from reindexing.

\end{document}